\documentclass[11pt, a4paper]{article}

\usepackage{amsmath,amssymb,amsthm,mathtools}
\usepackage[utf8]{inputenc}
\usepackage{mathrsfs}
\usepackage[mathcal]{euscript}
\usepackage[noblocks]{authblk}
\usepackage{hyperref}

\usepackage{xcolor}
\hypersetup{
	colorlinks   = true, 
	urlcolor     = {blue!90!black}, 
	linkcolor    = {blue!90!black}, 
	citecolor   = {red!90!black} 
}

\usepackage{cite,bm}

\allowdisplaybreaks[3]
\numberwithin{equation}{section}

\newcommand{\R}{\mathbb{R}}
\newcommand{\N}{\mathbb{N}}
\newcommand{\Z}{\mathbb{Z}}
\renewcommand{\d}{\mathrm{d}}

\renewcommand{\P}{\mathbb{P}}
\newcommand{\PP}[1]{\mathbb{P}\left\{#1\right\}}
\newcommand{\EE}[1]{\mathbb{E}\left\{#1\right\}}

\newcommand{\Pc}[2]{\mathbb{P}\left[#1 \middle| #2 \right]}

\newcommand{\E}{\mathbb{E}}
\renewcommand{\N}{\mathbb{N}}

\renewcommand{\b}{\mathbf{b}}

\theoremstyle{plain}
\newtheorem{thm}{Theorem}[section]
\newtheorem{prop}[thm]{Proposition}
\newtheorem{lem}[thm]{Lemma}
\newtheorem{cor}[thm]{Corollary}
\newtheorem{exa}[thm]{Example}

\theoremstyle{definition}

\newtheorem{con}{Condition}
\newtheorem{defi}[thm]{Definition}
\newtheorem{rmk}[thm]{Remark}

\author[1]{Viktor~Bezborodov\thanks{{\tt viktor.bezborodov@univr.it}.}}
\author[1]{Luca~Di~Persio\thanks{{\tt luca.dipersio@univr.it}.}}
\author[2]{Dmitri~Finkelshtein\thanks{{\tt d.l.finkelshtein@swansea.ac.uk}.}}
\affil[1]{Department of Computer Science, The University of Verona, Strada~le~Grazie 15, 37134 Verona, Italy}
\affil[2]{Department of Mathematics,
Swansea University, Bay Campus, Fabian Way, Swansea SA1 8EN, U.K.}
\author[3]{Yuri~Kondratiev\thanks{{\tt kondrat@math.uni-bielefeld.de}.}}
\author[3]{Oleksandr~Kutoviy\thanks{{\tt kutoviy@math.uni-bielefeld.de}.}}
\affil[3]{Fakult\"{a}t
f\"{u}r Mathematik, Universit\"{a}t Bielefeld, Postfach 110 131, 33501 Bielefeld, Germany}

\title{Fecundity regulation in a~spatial~birth-and-death~process}

\begin{document}

\maketitle

\begin{abstract}
	We study a Markov birth-and-death process on a space of locally finite configurations, which describes an ecological model with a density dependent fecundity regulation mechanism. We establish existence and uniqueness of this process and analyze its properties. In particular, we show global  time-space boundedness of the population density and, using a constructed Foster--Lyapunov-type function, we study return times to certain level sets of tempered configurations. We find also sufficient conditions that the degenerate invariant distribution is unique for the considered process.
	\end{abstract}



\section{Introduction}
Most of models in ecology are structured by space. Nowadays, individual based models in spatial ecology form a well established research area. We~refer for historical comments and detailed review to \cite{Ovaskainen+Co}. Mathematically, such models may be often described as Markov birth-and-death processes on configuration spaces over proper location sets.
A simple example is an independent birth process in a population $\gamma$ located in the Euclidean space $\R^\d$: each member of the population $\gamma$, placed at $y\in\gamma$, independently sends its off-spring to the target location $x\in\R^\d$ after an exponentially distributed random time. The displacement $x-y$ is chosen at random according to a certain dispersion kernel. 
The rate $\varkappa>0$ of the time-distribution is called the (density independent) \emph{fecundity} rate. Then, regardless of a particular dispersion kernel, the density of the population will grow exponentially in time.

The simplest regulation mechanism to prevent the unbounded growth is to include a (density independent) \emph{mortality} to the process. Namely, each member of the population may die after an exponentially distributed random time (independent from the birth time) with a rate $m>0$. Then, for $m< \varkappa$, the density of the system is still exponentially increasing in time, whereas $m=\varkappa$ is a critical value where the density is stabilized, and finally, for $m>\varkappa$, the density will exponentially decay showing the extinction of the population. This process can be treated also as a nonlocal branching process (cf. e.g. \cite[Section~4.3]{Li2011}), namely, each member dies after random time and, with certain rates, may produce $0$ or $2$ off-springs with the restriction that one of the latter is placed at parent's position, see \cite{Dur1979}. We will call this process \emph{spatial contact model} following \cite{KonSkor}, see also \cite{KonKutPirogov,KonPirogovZhizhina}. Note that, in the so-called critical case $m=\varkappa$, the considered system will have a unique invariant distribution for each dimension $d\geq 3$ (a probability measure over the space of configurations in $\R^\d$). This measure, in particular, has fast growth of factorial moments. 

A more sophisticated regulation mechanism is to consider a density dependent mortality rate. Such a rate is just the sum of the constant mortality and competitions with all other members of the population defined through a competition kernel. This describes the so-called \emph{spatial logistic model}, see \cite{Ovaskainen+Co} and references therein for its biological motivations. The corresponding Markov processes on finite sets (populations) in $\R^\d$ was analysed in \cite{FournierMeleard},  see also \cite{EK14}. Infinite populations were studied mathematically
in \cite{FinKoKut1,FinKoKut2} and \cite{KoKozitsy1,KoKozitsy3} in terms of the evolution of states
(measures) on the space of locally finite configurations. 
Such approach to study dynamics of infinite populations is known as statistical one. More precisely, the evolution of states in this approach is described by the evolution of the corresponding factorial moments (a.k.a. correlation functions). Note that a construction of the corresponding Markov process remains a challenging problem.

To the best of our knowledge, for locally finite systems in $\R^\d$, there are only few references concerning the construction of a general Markov birth-and-death process, namely, \cite{Garcia, GarciaKurtz,constantbirthrate}. In the present paper we study several models which belong to the considered in \cite{GarciaKurtz} case with a constant death rate and a birth rate with certain structural properties (see Section~\ref{sec:BaD} below for details); the process is constructed as a solution to stochastic equation. On the contrary, a birth-and-death process with constant birth rate and the aforementioned density dependent mortality was constructed in \cite{constantbirthrate} using a different approach which is based on a comparison with a Poisson random connection graph.

The main novelty of the present paper is that we consider a rather different regulation mechanism compared to the spatial logistic model. Namely, keeping a constant (density independent) mortality rate, we consider birth with a density dependent fecundity of the form
\begin{equation}\label{introfec}
	\varkappa(y,\gamma):=\varkappa e^{-l(y,\gamma)}, \qquad l(y,\gamma):=\sum_{z\in\gamma\setminus\{y\}}\varphi(y-z),
\end{equation}
where, we recall, $\gamma\subset\R^\d$ represents the locally finite set of positions of the population members.
In other words, the competition (described by the kernel~$\varphi$) within the population does not influence chances to die, but decreases chances of producing off-springs. We show (Proposition~\ref{b bounded}) that under minimal restrictions on the dispersion and the competition kernels the whole birth rate for the system (which includes summation over different $y\in\gamma$) remains globally bounded as a function of the existing configuration and the position for an off-spring. This allows to show the existence and uniqueness of the corresponding process (Theorem~\ref{thm:existuniqfec}). It is worth noting that we do not require any comparison between the dispersion and the competition kernels. We allow also a modification of \eqref{introfec} with (in the simplest model case)
\[
\varkappa(y,\gamma)=\varkappa e^{-l(y,\gamma)} \bigl(1+p\, l(y,\gamma)\bigr), \quad p\geq0.
\]
This modification is well motivated biologically in the case $p>1$ as then the competition starts to affect negatively the sending of off-springs only after the population becomes `well-developed' (i.e. the value of $s:=l(y,\gamma)$ becomes large enough and the unimodal function $\varkappa e^{-s}(1+ps)$ starts to decay). This corresponds to the so-called \emph{weak Allee effect}, see e.g. \cite{Shi2006}. Note that such generalization was considered also in \cite{FinKoKut4} in terms of the aforementioned statistical dynamics; the existence and properties of the corresponding Markov process
remained open since then.

In Section~\ref{sec:boundedbirth}, we study some general properties of processes with constant death and bounded birth rates (in addition to the fecundity model above, we consider two others, see Examples~\ref{ex:Glauber}--\ref{ex:establishment}). In particular, we prove that return times to sufficiently large level sets of configurations are exponentially integrable random variables (Proposition~\ref{prop:expmoments}). To show this, we introduce and study a  Forster--Lyapunov  function on the space of (tempered) configurations.

Finally, in Section~\ref{sec:uniqinv}, we introduce sufficient conditions for general sublinear birth rates (including, in particular, those for the fecundity model), which ensure the uniqueness of the degenerate invariant distribution for the considered process. 

\section{Birth and death processes on configuration space}\label{sec:BaD}

Let $\mathscr{B}(\R^\d)$ denote the Borel $\sigma$-algebra  over the Euclidean space $\R ^\d$,  $\d \in \N$.
	We study a birth-and-death process taking values in the space of locally finite configurations (discrete subsets) of $\R^\d$:
	\[
	\Gamma :=\bigl\{ \gamma \subset \Lambda \bigm\vert |\gamma \cap B| <
	+ \infty \text{ for any bounded } B \in \mathscr{B}(\R^\d)\bigr \}.
\]
Henceforth, $|\eta|$ denotes the number of points in a discrete finite set $\eta\subset\R^\d$.

Throughout the paper, we identify a configuration $\eta\in\Gamma_G$ with a discrete (counting) measure on $(\R^\d,\mathscr{B}(\R^\d))$ defined by assigning a unit mass to each atom at $x\in\eta$.

We fix an arbitrary $\varepsilon>0$ and consider the function
\begin{equation}\label{defG}
G(x):=(1+|x|)^{-\d-\varepsilon}, \quad x\in\R^\d,
\end{equation}
where $|x|$ denotes the Euclidean norm on $\R^\d$. We denote then
\begin{equation*}
\Gamma_G:=\Bigl\{\gamma\in\Gamma \Bigm\vert \langle G,\gamma\rangle:= \sum_{x\in\gamma} G(x)<\infty\Bigr\},
\end{equation*}
a set of \emph{tempered} configurations. We define a sequential topology on $\Gamma_G$ by assuming that $\gamma_n\to\gamma$, $n\to\infty$, if only
\[
	\lim_{n\to\infty}\langle f,\gamma_n\rangle=\langle f,\gamma\rangle
\]
for all $f\in C_b(\R^\d)$ (the space of bounded continuous functions on $\R^\d$) such that $|f(x)|\leq M_f G(x)$ for some $M_f>0$ and all $x\in\R^\d$. Let $\mathscr{B}_G(\Gamma)$ denote the corresponding Borel $\sigma$-algebra.

\begin{defi}\label{defBaD}
Let $b: \R ^\d \times \Gamma_G \to \R _+:=[0,\infty)$ be a measurable function. We describe a spatial birth-and-death process $\eta:\R_+\to\Gamma_G$ with the unit death rate and the birth rate $b$ through the following three properties:
\begin{enumerate}
	\item If the system is in a state $\eta_t \in \Gamma_G$
	at the time $t\in\R_+$, then the probability
	that a new particle appears (a ``birth'' happens) in
	a bounded set $B\in \mathscr{B}(\R ^\d)$
	during a time interval $[t;t+ \Delta t]$ is
	\begin{equation}\label{meaning_of_birthrate}
		\Delta t \int\limits _{ B}b(x, \eta)dx + o(\Delta t).
	\end{equation}
	\item If the system is in a state $\eta_t \in \Gamma_G$
	at the time $t\in\R_+$, then, for each $x\in\eta_t$, the probability
	that the particle at $x$ dies
	during a time interval $[t;t+ \Delta t]$ is $1\cdot \Delta t + o(\Delta t)$.
	\item With probability $1$ no two events described above happen simultaneously.
\end{enumerate}
\end{defi}

	\begin{rmk}
Following a convention for continuous-space processes, see e.g. \cite{GarciaKurtz},  we will say that the function $b$ is the \emph{birth rate} of the process, even though the rate of birth inside a bounded region $B \in \mathscr{B}(\R ^\d)$ is given by 
	 \[
	 \nu_\eta(B):=\int_B b(x, \eta)dx,
	 \] 
that is,  $\PP{|(\eta _t \setminus \eta _0) \cap B| =1 } = \nu_{\eta_0}(B) t + o(t)$. Thus, the function $b$ is actually   a version of the Radon--Nikodym derivative of the rate (considered as the measure $\nu_\eta$) w.r.t. the Lebesgue measure.
For the notion of the transition rates for interacting particle systems, see e.g. \cite[Chapter 1.3]{Lig85}.
	\end{rmk}

	The (heuristic) generator of our process is
	\begin{equation} \label{the generator}
	L F (\eta) =
	\sum\limits _{x \in \eta}  \bigl(F(\eta \setminus {x}) - F(\eta)\bigr) +\int\limits _{\R^\d} b(x, \eta) \bigl(F(\eta \cup {x}) -F(\eta)\bigr) dx.
	\end{equation}

	\begin{defi}\label{def:sol}
\begin{enumerate}
		\item Let $\tilde N$ be the Poisson point process on
	$\R_+  \times \R^\d \times \R_+ ^2 $ with
	mean measure $ds \times dx \times du \times e ^{-r} dr$.
The process $\tilde N$ is said to be compatible w.r.t. a filtration $\{\mathcal{F}_t\}$ if, for any measurable $A\subset \R^\d \times\R_+^2$, $\tilde N([0, t],A)$ is $\mathcal{F}_t$-measurable and $\tilde N((t, s],A)$ is independent of $\mathcal{F}_t$ for $0 < t < s$.
\item 	Let  $\eta _0$ be a $\Gamma_G$-valued $\mathcal{F}_0$-measurable random variable independent on $\tilde N$. Consider a point process $\tilde {\eta} _0$ on $\R^\d\times\R_+$ obtained by attaching to each point of $\eta_0$ an independent unit exponential random variable. Namely, if $\eta_0=\{x_i:i\in\N\}$ then
	$\tilde {\eta} _0=\{(x_i,\tau_i):i\in\N\}$ and $\{\tau_i\}$ are independent unit exponentials, independent of $\eta_0$ and $\tilde N$.
\item We will say that a process $(\eta_t)_{t\geq0}$ with sample paths in the Skorokhod space
$D_{\Gamma_G}[0,\infty)$ has the unit death rate and the birth rate $b$ if it is adapted to a filtration $\{\mathcal{F}_t\}$ w.r.t. to which $\tilde N$ is compatible and if, for any bounded $B\in\mathscr{B}(\R^\d)$, the following equality holds almost surely
	\begin{equation} \label{se GK}
	\begin{aligned}
	\eta _t (B) = \int\limits _{(0,t] \times B \times\R_+^2}
	& I _{ [0,b(x,\eta _{s-} )] } (u) I \{ r> t - s \}
	\tilde N(ds,dx,du,dr) \\
	+ \int\limits _{B \times \R_+ } &
	I \{ r > t \} \tilde {\eta} _0 (dx, d r )
	\end{aligned}
	\end{equation}
	where $\eta_t(B)=|\eta_t\cap B|$ is the number of points in $B$; note that henceforth we use configurations and counting measures interchangeably. 
	\end{enumerate}
	\end{defi}

\begin{rmk}
	We will sometimes denote the solution process at time $t$ as $\eta(t, \eta _0)$
	to underline the dependence on the initial condition.
	In other words, $\eta_t=\eta(t,\eta_0)$.
	As is the convention for Markov processes, we use notation
	$
	\E _{\eta _0}
	$ for the expectation related to the distribution of  $(\eta(t, \eta _0))_{t \geq 0}$.
\end{rmk}

\begin{defi}
Let $\Gamma\ni \eta=\{x_i\}$ and $y\in\R^\d$. We set $S_y\eta:=\{x_i-y\}\in\Gamma$. 
\begin{enumerate}
	\item The birth rate $b$ is said to be translation invariant if
	\[
		b(x+y,\eta)=b(x,S_y\eta), \quad x\in\R^\d, \eta\in\Gamma_G.
	\]
	\item A $\Gamma_G$-valued random variable $\eta$ is said to be  translation invariant if the distribution of $S_y\eta$ does not depend on $y$.
\end{enumerate}
\end{defi}

The following statement is a particular case of results from \cite{GarciaKurtz}.

\begin{thm}[cf. {\cite[Theorem~2.13, Lemma~3.14]{GarciaKurtz}}]\label{thm_exist_uniq}
Suppose that
\begin{equation}\label{b_is_bdd}
\mathbf{b}:=\sup_{\substack{x\in\R^\d\\\eta\in\Gamma_G}} b(x,\eta)<\infty,
\end{equation}
and, for some $M>0$,
\begin{equation}\label{main_cond_GK}
\sup_{\eta\in\Gamma_G}\bigl\lvert b(x,\eta\cup y) -b(x,\eta)\bigr\rvert\leq M G(x-y), \quad x,y\in\R^\d.
\end{equation}
Then there exists a unique solution to \eqref{se GK} in the sense of Definition~\ref{def:sol}. If,~additionally, both $b$ and $\eta_0$ are translation invariant, then $\eta_t$ is translation invariant for $t>0$.
\end{thm}
To prove the latter statement, we need the following simple lemma.
\begin{lem}
For any $x,y,z\in\R^\d$,
\begin{equation}\label{G_ineq}
	G(x-y)G(y-z)\leq G(x-z).
\end{equation}
\end{lem}
\begin{proof}
Immediately follows from \eqref{defG} and the inequality
\[
	1+|x-z|\leq 1+|x-y|+|y-z|\leq (1+|x-y|)(1+|y-z|).\qedhere
\]
\end{proof}
\begin{proof}[Proof of Theorem~\ref{thm_exist_uniq}]
We follow the ideas of \cite[Remark 4.1]{Xin}. First of all, \eqref{b_is_bdd} yields
\[
\sup_{\eta\in\Gamma_G}\int_{\R^\d}G(x)b(x,\eta)\,dx\leq \mathbf{b}
\int_{\R^\d}G(x)\,dx<\infty,
\]
that implies that \cite[Condition~2.1]{GarciaKurtz} is satisfied.

Next, by \cite[Lemma~2.15]{GarciaKurtz}, \eqref{main_cond_GK} implies that, for any $\eta_1,\eta_2\in\Gamma_G$ and $x\in\R^\d$,
\[
	\bigl\lvert b(x,\eta_1)-b(x,\eta_2)\bigr\rvert\leq M\int_{\R^\d}G(x-y)|\eta_1-\eta_2|(dy),
\]
where $|\eta_1-\eta_2|$ means the total variation norm of the (signed) discrete measure $\eta_1-\eta_2$ on $(\R^\d,\mathscr{B}(\R^\d))$. By using \eqref{G_ineq} with $z=0$ (and swapping $x$ and $y$), we get that, for each $x\in\R^\d$,
\[
	\bigl\lvert b(x,\eta_1)-b(x,\eta_2)\bigr\rvert\leq M\bigl(1+|x|\bigr)^{d+\varepsilon}\int_{\R^\d}G(y)|\eta_1-\eta_2|(dy),
\]
that implies that \cite[Condition~2.2]{GarciaKurtz} is satisfied. 

The last assumption of \cite[Theorem~2.13]{GarciaKurtz} can be read in our settings as follows
	\[
	\sup_{x}\int_{\R^\d} \frac{c(x)MG(x-y)}{c(y)}dy<\infty
	\]
	for some positive bounded (cf.~\cite[Remarks~4.1(a)]{Xin}) function $c$; the latter inequality evidently holds with $c\equiv 1$. As a result, one gets the statement from \cite[Theorem~2.13, Lemma~3.14]{GarciaKurtz}.
\end{proof}

\section{Description of the model}
We consider a birth-and-death process on $\Gamma_G$ with the unit death rate and the birth rate given, for some $0\leq a,c,\varphi\in L^1(\R^\d)$, by
	\begin{equation}\label{brate fecundity}
	b(x,\eta) = \sum\limits _{y \in \eta} a(x-y)
\biggl(1+\sum\limits _{z \in \eta \setminus \{y\}}
	c(z-y)\biggr)
	\exp\biggl(-\sum\limits _{z \in \eta \setminus \{y\}}
	\varphi(z-y)\biggr)	.
	\end{equation}

In view of \eqref{meaning_of_birthrate}, the rate \eqref{brate fecundity} can be interpreted as follows. Let $\varkappa=\int_{\R^\d} a(x)\,dx>0$. If the system is in a state $\eta \in \Gamma_G$ at the time $t\in\R_+$, then each $y\in\eta$ may send an off-spring after exponential random time whose rate is $\varkappa r(y,\eta)$,  where
\begin{equation}\label{r_rate}
r(y,\eta):=\biggl(1+\sum\limits _{z \in \eta \setminus \{y\}}
	c(z-y)\biggr)
	\exp\biggl(-\sum\limits _{z \in \eta \setminus \{y\}}
	\varphi(z-y)\biggr).
\end{equation}
The off-spring will be sent according to the probability distribution on $\R^\d$ with the (normalized) density $\varkappa^{-1}a$, i.e. the probability that the off-spring (sent from $y\in\eta_t$) appears in a bounded $B\in\mathscr{B}(\R^\d)$ is $\varkappa^{-1}\int_B a(x-y)\,dx$. Note that we allow $c=0$ a.e. 

In ecology, the rate $\varkappa r(y,\eta)$ is called the \emph{fecundity}. A model example is the case when $c(z)=p\varphi(z)$, $z\in\R^\d$, for some $p\geq 0$. Then \eqref{r_rate} is just the value of $R_p(s):=(1+ps)e^{-s}$, $s\in\R_+$, at $s=\sum\limits _{z \in \eta \setminus \{y\}} \varphi(z-y)$. The function $R_p$ is decreasing on $\R_+$ for $p\in[0,1]$ and is \emph{unimodal} for $p>1$, i.e. it has a unique maximum point; at $s=\frac{p-1}{p}$. As a result, the case $p\in[0,1]$ describes the model where the `wish' for $y\in\eta_t$ to send an off-spring decays because of other particles around $y$. Whereas the  case $p>1$ describes the so-called \emph{weak Allee effect} when the small density of the system around $y$ increases the chances for an off-spring to be sent, but there exists a threshold for that density after which the surrounders of $y$ decrease the chances.

	The following lemma is the key tool in proving the global boundedness of the birth rate \eqref{brate fecundity}.

	\begin{lem}\label{lem1dot4}
		Let $\alpha, \rho>0$, and let  $b_f:\R_+\to(0,\infty)$ be a bounded decreasing to $0$ on $\R_+$  function, such that
		\begin{equation}\label{v0}
		\int_{\R_+} b_f(s)s^{\d-1}\,ds<\infty.
		\end{equation}
		Let $f,g:\R^\d\to\R_+$ be measurable functions, $f$ is bounded, such that
		\begin{alignat}{2}
		&f(x)\leq b_f(|x| ), &&\qquad x\in\R^\d,\label{v1}\\
		&g(x)\geq \alpha, &&\qquad |x|\leq \rho.\label{v2}
		\end{alignat}
		Then
		\begin{equation*}
		\sup\limits _{\substack{x \in \R ^\d\\\eta \in \Gamma}}
		\sum\limits _{y \in \eta} f(x-y)\exp\Bigl\{-\sum\limits _{z \in \eta \setminus y} g(z-y) \Bigr\} < \infty.
		\end{equation*}
	\end{lem}
	\begin{proof}
		For each $x=(x_1,\ldots,x_d)\in\R^\d$, let $|x|$ denote the Euclidean norm of $x$, and let $|x|_\infty :=\max\limits_{1\leq i\leq \d}|x_i|$. We have then
		\begin{equation}\label{clearly}
		|x|_\infty\leq |x|\leq \sqrt{\d}|x|_\infty, \quad x\in\R^\d.
		\end{equation}
		Set $q:=\frac{\rho}{\sqrt{\d}}>0$, and, for each $\bm{k}=(k_1,\ldots,k_d)\in\Z^\d$, consider a cube
		\[
		H_q(\bm{k})=\Bigl\{ x\in\R^\d \Bigm\vert q\Bigl(k_i-\frac{1}{2}\Bigr) < x_k\leq q\Bigl(k_i+\frac{1}{2}\Bigr) \Bigr\}
		\]
		centered at $q \bm{k}\in\R^\d $ with edges of the length  $q$.   Note that $\R^\d=\bigsqcup_{\bm{k}\in\Z^\d} H_q(\bm{k})$.

		For each $z \in\R^\d$ with $|z|_\infty\leq q$, we have, by \eqref{clearly}, $|z|\leq \rho$, and hence, by \eqref{v2}, $g(z)\geq \alpha$. Therefore, for $x,y \in H_q(\bm{k})$, $\bm{k} \in \Z ^\d$,  we have $|x-y|_{\infty} \leq q$ and hence 
		\begin{equation*}
		g(x-y)\geq \alpha, \quad x,y\in H_q(\bm{k}), \ \bm{k}\in\Z^\d.
		\end{equation*}
		Furthermore, for $ x\in H_q(\bm{k} _x),  y\in H_q(\bm{k} _y) $, $ \bm{k}_x, \bm{k}_y\in\Z^\d $, we have
		$|x - q\bm{k}_x| \leq \frac{\rho}{2}$, $|y - q\bm{k}_y| \leq \frac{\rho}{2}$,
		and 
		\begin{equation*}
		|x-y| \geq -|x - q\bm{k}_x| + q|\bm{k}_x - \bm{k}_y| - | q\bm{k}_y - y| \geq  q |\bm{k}_x - \bm{k}_y| - \rho,
		\end{equation*}
		and hence by \eqref{v1} and the monotonicity of $b_f$
		\begin{equation*}
		f(x - y) \leq  b_f(|x - y|  ) \leq b_f\Bigl( (q |\bm{k}_x - \bm{k}_y| - \rho)\vee 0 \Bigl).
		\end{equation*}

		As a result, for any $\eta\in\Gamma$,
		\begin{gather*}
		\sum _{y \in \eta} \exp\biggl(-\sum_{z \in \eta \setminus \{y\}} g(z-y)\biggr)
		f(x-y)
		\\
		=\sum_{\bm{k}\in\Z^\d}\sum _{y \in \eta\cap  H_q(\bm{k})} \exp\biggl(-\sum_{\bm{j}\in\Z^\d}\sum_{z \in (\eta\cap  H_q(\bm{j})) \setminus \{y\}} g(z-y)\biggr)f(x-y) \\
		\leq \sum_{\bm{k}\in\Z^\d}\sum _{y \in \eta\cap  H_q(\bm{k})} \exp\biggl(-\sum_{z \in (\eta\setminus \{y\})\cap  H_q(\bm{k})} g(z-y)\biggr)f(x-y)\\
		\leq
		\sum_{\substack{\bm{k}\in\Z^\d}}\sum _{y \in \eta\cap  H_q(\bm{k})} \exp\bigl(-\alpha \bigl\lvert ( \eta\setminus \{y\})\cap  H_q(\bm{k})\bigr\rvert\bigr)b_f\Bigl((q|\bm{k} -\bm{k}_x| - \rho)\vee 0\Bigr)\\
		=
		\sum_{\substack{\bm{k}\in\Z^\d}}\bigl\lvert  \eta\cap  H_q(\bm{k})\bigr\rvert  \exp\bigl(-\alpha \bigl\lvert  \eta\cap  H_q(\bm{k})\bigr\rvert\bigr) e^{\alpha} b_f\Bigl((q|\bm{k} -\bm{k}_x| - \rho)\vee 0\Bigr)\\
		\leq \frac{e^{\alpha}}{\alpha e}
		\sum_{\substack{\bm{k}\in\Z^\d}}   b_f\Bigl((q|\bm{k} -\bm{k}_x| - \rho)\vee 0\Bigr)=
		\frac{e^{-\alpha}}{\alpha e}
		\sum_{\bm{j}\in\Z^\d}   b_f\Bigl((q|\bm{j} | - \rho)\vee 0\Bigr)<\infty,
		\end{gather*}
		because of \eqref{v0}, and the bound does not depend on $x$.
		In the penultimate inequality we used that $s e^{-\alpha s} \leq \frac {1}{\alpha e}$ for $s \geq 0$.
		The lemma is proved.
	\end{proof}

To show the existence of the process (see Theorem~\ref{thm:existuniqfec} below), we will require the following assumptions.
	\begin{con}\label{con1}
		There exists $B\geq 1$, such that, for a.a. $x\in\R^\d$,
		\begin{equation}\label{dima_cond}
		a(x)\leq B G^2(x), \qquad \varphi(x)\leq  B G(x).
		\end{equation}
	\end{con}

	\begin{con}\label{varphi sep from 0}
		The function $\varphi$ is separated from $0$ in a neighborhood of the origin.
	\end{con}

	\begin{con}\label{con3}
	There exists $p\geq0$, such that $c(x)\leq p \varphi(x)$ for a.a. $x\in\R^\d$.
	\end{con}

	\begin{prop} \label{b bounded}
		Let the birth rate $b$ is given by \eqref{brate fecundity}.
Suppose that there exists $B\geq 1$ such that, for a.a. $x\in\R^\d$,
		\begin{equation*}
		a(x)\leq B G(x), \qquad \varphi(x)\leq  B G(x)
		\end{equation*}
(in particular, let Conditions~\ref{con1} hold). Suppose also that Conditions~\ref{varphi sep from 0}--\ref{con3} hold. Then $b$ is uniformly bounded, i.e. \eqref{b_is_bdd} holds.
	\end{prop}
	\begin{proof}
Condition~\ref{con3} implies that
\begin{align}
b(x,\eta)&\leq \sum\limits _{y \in \eta} a(x-y)
	\exp\biggl(-\frac{1}{2}\sum\limits _{z \in \eta \setminus \{y\}}
	\varphi(z-y)\biggr) \notag\\&\quad \times
\biggl(1+p\sum\limits _{z \in \eta \setminus \{y\}}
	\varphi(z-y)\biggr)
	\exp\biggl(-\frac{1}{2}\sum\limits _{z \in \eta \setminus \{y\}}
	\varphi(z-y)\biggr)\notag\\
	&\leq C_p \sum\limits _{y \in \eta} a(x-y)
	\exp\biggl(-\frac{1}{2}\sum\limits _{z \in \eta \setminus \{y\}}
	\varphi(z-y)\biggr),\label{added}
\end{align}
where
\[
C_p=\sup_{s\geq0}(1+ps)e^{-\frac{1}{2}s}=\begin{cases}
1, & 0\leq p\leq \frac{1}{2},\\
2p e^{\frac{1}{2p}-1}, & p>\frac{1}{2}.
\end{cases}
\]
By Condition~\ref{varphi sep from 0} and the first inequality in \eqref{dima_cond}, one can apply 
	Lemma \ref{lem1dot4}  with $f = a$, $g = \frac12 \varphi$, $b_f(s)=B(1+s)^{-\d-\varepsilon}$ to \eqref{added}, that yields the statement.
	\end{proof}

	\begin{thm}\label{thm:existuniqfec}
	Let  $b$ be the birth rate given by \eqref{brate fecundity} and Conditions~\ref{con1}--\ref{con3} hold. Then there exists a unique solution to \eqref{se GK} in the sense of Definition~\ref{def:sol}. If, additionally, both $b$ and $\eta_0$ are translation invariant, then $\eta_t$ is translation invariant for $t>0$.
	\end{thm}

	\begin{proof}
	We apply Theorem~\ref{thm_exist_uniq}. Since Proposition~\ref{b bounded} implies \eqref{b_is_bdd}, it is enough to check that \eqref{main_cond_GK} holds. 	For  $\eta\in\Gamma_G$, $x\in\R^\d$ and $x'\in\R^\d\setminus \{x\}$, we have from \eqref{brate fecundity} and \eqref{r_rate},
	\begin{align*}
&\quad 	b(x,\eta\cup x') -b(x,\eta) \\&= \sum\limits _{y \in \eta} a(x-y)
 e^{-\varphi(x'-y)}r(y,\eta) \\&+\sum\limits _{y \in \eta} a(x-y)
c(x'-y) e^{-\varphi(x'-y)}
	\exp\biggl(-\sum\limits _{z \in \eta \setminus \{y\}}
	\varphi(z-y)\biggr)\\
	&\quad+ a(x-x')r(x',\eta)-\sum\limits _{y \in \eta} a(x-y)
r(y,\eta).
	\end{align*}
Then, by using  inequalities $\bigl\lvert e^{-\varphi}-1\bigr\rvert \leq \varphi $ and $0\leq c\leq p \varphi$, we get
	\begin{align*}
	&\quad \bigl\lvert b(x,\eta\cup x') -b(x,\eta)\bigr\rvert \\ &  \leq \sum\limits _{y \in \eta} a(x-y) \varphi(x'-y) (1+p) r(y,\eta)\\
	&\quad+ a(x-x')
\biggl(1+p\sum\limits _{z \in \eta }
	\varphi (z-x')\biggr) 	\exp\biggl(-\sum\limits _{z \in \eta }
	\varphi(z-x')\biggr).
	\end{align*}
	Next, by \eqref{dima_cond} and \eqref{G_ineq},
	\[
	a(x-y) \varphi(x'-y)\leq B^2 G^2(x-y) G(x'-y)\leq B^2 G(x-y) G(x-x').
	\]
	As a result,
	\begin{align*}
	&\quad \bigl\lvert b(x,\eta\cup x') -b(x,\eta)\bigr\rvert \\& \leq
	(1+p)B^2 G(x-x') \sum\limits _{y \in \eta} G(x-y) r(y,\eta) + B G^2(x-x') \sup_{s\geq 0}(1+ps)e^{-s}\\
	&\leq \mathrm{const}\cdot G(x-x'),
	\end{align*}
	where we used Proposition~\ref{b bounded} with $a$ replaced by $G$.

	Hence \eqref{main_cond_GK} holds, and we get the statement from Theorem~\ref{thm_exist_uniq}.
	\end{proof}

\begin{rmk}
It is straightforward to check, following the proofs above, that the statements of Proposition~\ref{b bounded} and Theorem~\ref{thm:existuniqfec} remain true if we replace $a(x-y)$ in \eqref{brate fecundity} by $a_1(x)a_2(y)a(x-y)$ with $0\leq a_1,a_2\in L^\infty(\R^\d)$.
\end{rmk}

	\section{Properties of a process with bounded birth rate}\label{sec:boundedbirth}

In this Section, we study some general properties of birth-and-death processes which are described by Definition~\ref{defBaD} and which have globally bounded birth rate $b$. Namely, let the assumptions \eqref{b_is_bdd}--\eqref{main_cond_GK} hold and then, by Theorem~\ref{thm_exist_uniq}, $\eta:\R_+\to\Gamma_G$ is the unique solution to \eqref{se GK} in the sense of Definition~\ref{def:sol}.

One example of such rate given by \eqref{brate fecundity} under Conditions~\ref{con1}--\ref{con3} was discussed in Theorem~\ref{thm:existuniqfec}. Consider another examples. 
\begin{exa}[Glauber dynamcis in continuum]\label{ex:Glauber}
Consider the rate
\[
	b_{z,\phi}(x,\eta)=z \exp\biggl(-\sum\limits _{y \in \eta \setminus \{x\}}
	\phi(x-y)\biggr),
\]
where $z>0$ and $\phi:\R^\d\to\R_+$ is such that $\phi(x)\leq B G(x)$, $x\in\R^\d$ for some $B>0$. Then the mapping \eqref{the generator} is the generator of the so-called Glauber dynamics in continuum which was actively studied in recent decades, see e.g. \cite{FKKoz2011,KL2005,FinKoKut1} and references therein. Clearly, since $\phi\geq0$, the assumption \eqref{b_is_bdd} is satisfied. Next,
\[
	\bigl\lvert b_{z,\phi}(x,\eta\cup x')-b_{z,\phi}(x,\eta)\bigr\rvert =b_{z,\phi}(x,\eta)\bigl(1-e^{-\phi(x-x')}\bigr)\leq z\phi(x-x'),
\]
and hence \eqref{main_cond_GK} holds as well.

An important particular case is when $\phi\equiv0$, i.e. both death and birth rates are constants. The corresponding process is called a \emph{Surgailis process}, cf.~\cite{Sur84}.
\end{exa}

\begin{exa}[Establishment rate]\label{ex:establishment}
Consider the rate, cf. \eqref{brate fecundity}:
	\begin{equation}\label{brate extablishment}
	b_{a,c,\phi}(x,\eta) = \sum\limits _{y \in \eta} a(x-y)
\biggl(1+\sum\limits _{z \in \eta}
	c(x-z)\biggr)
	\exp\biggl(-\sum\limits _{z \in \eta }
	\phi(x-z)\biggr)	,
	\end{equation}
where $0\leq a,c,\phi\in L^1(\R^\d)$ are such that, for $x\in\R^\d$,
\begin{equation}\label{estkernest}
a(x)\leq q \phi(x), \qquad c(x)\leq p \phi(x), \qquad \phi(x)\leq B G(x),
\end{equation}
for some $q,B>0$, $p\geq0$. Here if the system is in a state $\eta_t \in \Gamma_G$ at the time $t\in\R_+$, then each $y\in\eta_t$ may send an off-spring after exponential random time whose rate is $\langle a \rangle$. The off-spring will be sent according to the probability distribution on $\R^\d$ with the (normalized) density $\varkappa^{-1}a$. However, this off-spring may not survive because of a competition around it. The rate of surviving at $x$ is
\begin{equation}\label{survrate}
\biggl(1+\sum\limits _{z \in \eta}
	c(x-z)\biggr)
	\exp\biggl(-\sum\limits _{z \in \eta }
	\phi(x-z)\biggr).
\end{equation}

The assumptions in \eqref{estkernest} imply that
	$b_{a,c,\phi}(x,\eta)\leq g\Bigl( \sum\limits _{z \in \eta }
		\phi(x-z)\Bigr)$,
where $g(s)=qs(1+ps)e^{-s}$, $s\in\R_+$, and hence \eqref{b_is_bdd} holds. Moreover,
\begin{align*}
&\quad b_{a,c,\phi}(x,\eta\cup x')\\&=b_{a,c,\phi}(x,\eta)e^{-\phi(x-x')}\\ &\quad+a(x-x')e^{-\phi(x-x')}\biggl(1+\sum\limits _{z \in \eta}
	c(x-z)\biggr)
	\exp\biggl(-\sum\limits _{z \in \eta }
	\phi(x-z)\biggr)\\
	&\quad +a(x-x')e^{-\phi(x-x')}
	c(x-x')
	\exp\biggl(-\sum\limits _{z \in \eta }
	\phi(x-z)\biggr)\\&\quad +
	\sum\limits_{y \in \eta }  a(x-y)e^{-\phi(x-x')} 
	c(x-x')
	\exp\biggl(-\sum\limits _{z \in \eta } 
	\phi(x-z)\biggr).
\end{align*}
Then \eqref{estkernest} implies
\begin{multline*}
 \bigl\lvert b_{a,c,\phi}(x,\eta\cup x')-b_{a,c,\phi}(x,\eta)\bigr\rvert\\
\leq b_{a,c,\phi}(x,\eta)\bigl\lvert e^{-\phi(x-x')} -1\bigr\rvert
+2a(x-x')\sup_{s\in\R_+}(1+ps)e^{-s}\\ + q c(x-x') \sup_{s\in\R_+}se^{-s}
\leq B_1 G(x-x')
\end{multline*}
for some $B_1>0$, that yields \eqref{main_cond_GK}.
\end{exa}

\begin{prop}
Let \eqref{b_is_bdd}--\eqref{main_cond_GK} hold, and let
$\eta_t$ be the unique solution to \eqref{se GK} in the sense of Definition~\ref{def:sol}. Then there exists a Surgailis process $\xi_t$ with the unit death rate and the birth rate $\mathbf{b}$ such that $\xi _0 = \eta _0$ a.s. implies
		\begin{equation}\label{stoch_dom}
		\eta _t \subset \xi _t \quad\text{a.s.,} \ t>0.
		\end{equation}
		In other words,  $(\eta _t)_{t\geq0}$ is stochastically dominated by
		the Surgailis process $(\xi _t)_{t\geq0}$.
	\end{prop}
\begin{proof}
The process $\xi _t$ with the unit death rate and the constant birth rate $b(x,\eta)\equiv\mathbf{b}$ evidently satisfies the assumptions \eqref{b_is_bdd}--\eqref{main_cond_GK}, and hence, by Theorem~\ref{thm_exist_uniq}, $\xi _t$ is the unique solution to
	\begin{equation} \label{Surg}
	\begin{aligned}
	\xi _t (B) = \int\limits _{(0,t] \times B \times \R_+^2}
	& I _{ [0,\mathbf{b}] } (u) I \{ r> t - s \}
	\tilde N(ds,dx,du,dr) \\
	+ \int\limits _{B \times \R_+ } &
	I \{ r > t \} \tilde {\eta} _0 (dx, d r ).
	\end{aligned}
	\end{equation}

	Fix some $t>0$. Then, by \eqref{se GK}, a.s. for an $x \in \eta _t \setminus \eta _0$ there exist $s,u, r \in (0,t] \times \R _+^2$ such that
	\begin{equation}\label{icky}
	(s,x,u ,r) \in \tilde N, \ \ u \leq b(x, \eta _{s-}), \ \ r > t -s.
	\end{equation}
	Since $ b(x, \eta _{s-}) \leq \mathbf{b}$,
	 \eqref{Surg} and \eqref{icky} imply
	$x \in \xi _t$. Similarly, it follows from \eqref{se GK} and \eqref{Surg}
	that if $x \in \eta _t \cap \eta _0$, then also $x \in \xi _t \cap \eta _0$.
	 Therefore, \eqref{stoch_dom} holds.
\end{proof}

\begin{cor}
		There exists $C > 0 $ such that, for a bounded $B \in\mathscr{B}(\R ^\d)$,
		\begin{equation}\label{subvol}
		\EE{ |\eta _t \cap B| } \leq C \, \mathrm{vol} (B)
		\end{equation}
for all $t>0$, provided that \eqref{subvol} holds for $t=0$.
\end{cor}
\begin{proof}
Indeed, by \eqref{Surg},
\begin{align*}
\EE{|\xi_t\cap B|}& \leq \mathrm{vol} (B) \mathbf{b} \int_0^t\int_{t-s}^\infty e^{-r}dr ds+ \EE{|\eta_0\cap B|}\int_t^\infty e^{-r}dr\\
&=\mathrm{vol} (B)\mathbf{b}(1-e^{-t})+ \EE{|\eta_0\cap B|}e^{-t}\\&\leq \max\bigr\{\mathrm{vol} (B)\mathbf{b}, \EE{|\eta_0\cap B|}\bigr\},
\end{align*}
that implies the statement because of \eqref{stoch_dom}.
\end{proof}

For a function $0\leq f\in L^1(\R^\d)$, we will use the notation
\[
	\langle f\rangle:=\int_{\R^\d} f(x)dx<\infty.
\]

		Let a non-increasing function $K: (0,\infty) \to (0,\infty)$ be such that
		\begin{gather*}
			\lim\limits _{q \to 0+} K(q) = \infty, \\
			\int\limits_{r}^\infty K(q)q^{\d-1} dq < \infty, \quad r>0,
		\end{gather*}
		and let functions $\phi,h:\R^\d\to(0,\infty)$  be such that
		\begin{gather}
			C_1:=\sup_{x\in\R^\d }\int\limits _{\R ^\d} \phi(y)K(|x-y|) dy <\infty,\label{profligate}\\
			2C_1 \b \phi (x) \leq h(x)\leq G(x), \quad x\in\R^\d.\label{bland}
		\end{gather}
We will also assume $\phi$ and $h$ are separated from $0$ on each compact subset of $\R^\d$.

We define
	\[
	\begin{aligned}[2]
	\psi(x,y) &= \phi(x) \phi(y) K(|x-y|), &&\quad x, y \in \R ^ \d, x\neq y,\\
	V(\eta) &= \sum\limits _{\{x,y\}\subset \eta} \psi (x,y). &&\quad \eta\in\Gamma,
	\end{aligned}
	\]
and consider the mapping
	\begin{equation}\label{defW}
	\Gamma \ni \eta \mapsto W(\eta) := \langle h, \eta \rangle + V(\eta) \in
	[0,\infty],
	\end{equation}
where, we recall, $\langle h, \eta \rangle =\sum_{x\in\eta} h(x)$. 
The assumption \eqref{bland} implies that $\langle h,\eta \rangle<\infty$ for all $\eta\in\Gamma_G$. We set
\[
	\Theta:=\Theta_G:= \{\eta \in \Gamma_G : V(\eta) < \infty\}
\]

We are going to show now that if $\E  W (\eta _{0}) < \infty$ (and hence a.s. $\eta_0\in \Theta$) then a.s. $\eta_t\in \Theta$ for all $t\geq0$.

 \begin{prop}\label{propnew}
 Suppose that $\E W(\eta_0)<\infty$. Then $\E W(\eta_t)<\infty$, $t>0$. 
 \end{prop}
 \begin{proof}
Consider the Poisson process $(\Pi _t)_{t\geq0}$ defined by
	\begin{equation} \label{ppp}
	\begin{aligned}
	\Pi _t (B) = \int\limits _{(0,t] \times B \times \R_+^2}
	I _{ [0,\mathbf{b}] } (u)
	\tilde N(ds,dx,du,dr) .
	\end{aligned}
	\end{equation}
By \eqref{Surg} and \eqref{ppp}, a.s.
\begin{alignat}{2}
\xi _t \setminus \eta_0 &\subset \Pi _t, &&\quad t\geq0,\label{incl1}\\
\Pi _s &\subset \Pi _t, &&\quad 0\leq s \leq t.\label{incl2}
\end{alignat}
Next, for each $t\geq0$, $\Pi _t$ is a Poisson point process (a.k.a. Poisson random point field or Poisson random measure) on $\R ^\d$ with the intensity $\mathbf{b} t$.  Then, by \eqref{bland} and the Slivnyak--Mecke theorem,
	\[
	 \E \langle h, \Pi_t \rangle \leq \E \langle G, \Pi_t \rangle  = \mathbf{b} t \langle G\rangle < \infty, \quad t\geq0.
	\]
In particular, a.a. realisations of $\Pi_t$ lie in $\Gamma_G$ for $t\geq0$.
Similarly, for $t\geq0$,
	\[
	 \E V (\Pi_t) = (\mathbf{b} t)^2\int\limits _{\R ^\d} \int\limits _{\R ^\d} \phi (x) \phi (y) K(|x-y|) dx dy\leq C_1 \mathbf{b}^2 t^2\langle\phi\rangle ,
	\]
	where we used \eqref{profligate}. As a result, 
\begin{equation}\label{expWPifinite}
\E W (\Pi_t)<\infty, \quad t\geq0,
\end{equation}
and hence $\Pi_t\in\Theta$ a.s. for $t\geq0$. 

Next, by \eqref{incl1} and \eqref{stoch_dom}, we have  
\[
	\eta_t\subset\eta_0\cup\Pi_t, \quad t>0,
\]
and hence, by \eqref{defW},
\[
	W(\eta_t)\leq W(\eta_0\cup\Pi_t)=W(\eta_0)+W(\Pi_t)+\sum_{\substack{x\in\eta_0\\y\in\Pi_t}}\psi(x,y).
\]
Since $\eta_0$ and $\Pi_t$ are independent, we get, for all $t>0$,
\begin{align}
\E W(\eta_t)&\leq \E W(\eta_0)+\E W(\Pi_t)+\E \E\biggl[\sum_{\substack{x\in\eta_0\\y\in\Pi_t}}\psi(x,y) \biggm\vert \eta_0\biggr]\notag
\\&=\E W(\eta_0)+\E W(\Pi_t)+\mathbf{b}t \E \sum_{x\in\eta_0}\int_{\R^\d}\psi(x,y)dy \notag\\& \leq \E W(\eta_0)+\E W(\Pi_t)+C_1 \mathbf{b}t \E \sum_{x\in\eta_0}\phi(x)\notag\\& \leq \E W(\eta_0)+\E W(\Pi_t)+\frac{t}{2} \E W(\eta_0)<\infty,\label{usedlatter} 
\end{align}
where we used again the Slivnyak--Mecke theorem and also \eqref{profligate}, \eqref{bland}, \eqref{expWPifinite}.
\end{proof}

\begin{rmk}
We will show below (see Theorem~\ref{prop1}) a more stronger statement, namely, that $\limsup\limits _{t \to \infty} \E W(\eta _t)$ is finite. To this end, one needs to justify further properties of the process $(\eta_t)_{t\geq0}$.
\end{rmk}

\begin{lem}\label{lemabsLW}
For each  $\eta \in \Theta$,
	\begin{equation*}
		\bigl\lvert L W (\eta)\bigr\rvert\leq  \textbf{b}  \langle h\rangle+2 W(\eta)<\infty.
	\end{equation*}
\end{lem}
\begin{proof}
Using the equality
	\begin{equation}\label{twinge1}
	W(\eta \cup x) - W(\eta) =
	h(x) + \sum\limits _{y \in \eta} \psi(x,y), \ \ \ \eta \in \Theta, x \notin \eta,
	\end{equation}
and
	\eqref{b_is_bdd}, \eqref{profligate}, \eqref{bland}, we get, for all $\eta\in\Gamma_G$,
	\begin{align}
	&\quad \biggl\lvert\int\limits _{\R^\d} b(x, \eta) \bigl(W(\eta \cup {x}) -W(\eta)\bigr) dx\biggr\rvert \notag
	\\
	&\leq
	\textbf{b} \langle h\rangle   + \textbf{b} \sum\limits _{y \in \eta} \phi (y) \int\limits _{\R ^\d} \phi(x)K(|x-y|) dx\notag
	\\ &\leq
	\textbf{b}  \langle h\rangle + \b C_1 \sum\limits _{y \in \eta} \phi  (y)  \leq \textbf{b}  \langle h\rangle  + \sum\limits _{y \in \eta} h (y).\label{expiate}
	\end{align}
Next, using the equality
\begin{equation}\label{twinge11}
W(\eta \setminus {x}) - W(\eta)=-h(x)-\sum_{y\in\eta\setminus x} \psi(x,y), \quad \eta\in\Theta, \ x\in\eta,
\end{equation}
we have, for all $\eta\in\Theta$,
	\begin{equation}\label{expiate1}
	\biggl\lvert \sum\limits _{x \in \eta}  (W(\eta \setminus {x}) - W(\eta))\biggr\rvert
= \sum\limits _{x \in \eta} h (x) + 2 \sum\limits _{\{x,y\} \subset \eta} \psi(x,y).
	\end{equation}
Combining \eqref{expiate} with \eqref{expiate1}, we get the statement.
\end{proof}
\begin{cor}\label{slink}
		Assume that $\E W(\eta _0) < \infty$.
		Then 
		\begin{equation}\label{eq:Iamformal}
		\E \int\limits_0^t \bigl \lvert LW (\eta _{s-})\bigr\rvert ds < \infty, \quad t\geq0.
		\end{equation}
\end{cor}
\begin{proof}
	By Proposition~\ref{propnew}, a.s. $\eta_t\in\Theta$ for all $t>0$. Then, by Lemma~\ref{lemabsLW}, to prove \eqref{eq:Iamformal}, it is enough to show that 
	$\E \int\limits_0^t W (\eta _{s-}) ds $ is finite. The latter expression is estimated, because of \eqref{usedlatter}, by
\begin{align*}
\Bigl(t+\frac{t^2}{4}\Bigr)\E W(\eta_0) + \E \int\limits_0^t  W (\Pi _{s-}) ds\leq \Bigl(t+\frac{t^2}{4}\Bigr)\E W(\eta_0) + t\,\E W (\Pi _{t})<\infty,
\end{align*}
where we used \eqref{incl2} and \eqref{expWPifinite}.
\end{proof}

Define $N _b$ as the projection of $\tilde N$ on first, second, and fourth coordinates. Then, in particular, 
\begin{equation*}
N_b (\{(x,s,u)\}) = 1       \Longleftrightarrow \exists r >0 : \tilde N(\{(x,s,r,u)\}) = 1.
\end{equation*}
Since $\int\limits _{\R _+} e^{-r}dr = 1$, $N _b$ is a Poisson point process on $\R_+ \times \R ^\d \times \R _+$
with mean measure $ds \times dx \times du$.

 Let $\mathscr{B}_b(\R^\d)$ denote the set of all bounded Borel subsets of $\R^\d$.
 Define, for $t \geq 0$,
	\begin{multline*}
	\mathscr{F}_t ^o := \sigma\Bigl( N_b(A \times U) , \eta _{s}(B) \Bigm\vert A \in \mathscr{B}((0,t]),
	\\  U \in \mathscr{B}(\R ^\d \times \R _+ ), 
	s \in [0,t]\cap \mathbb{Q}, B \in \mathscr{B}_b(\R ^\d)  \Bigr),
	\end{multline*}
	and let  $\mathscr{F}_t$ be the completion of $\mathscr{F}_t ^o $ under $\P$.
	Then $ N_b$ is compatible with~$\mathscr{F}_t$.

  For $B \in \mathscr{B}_b(\R^\d)$, define $D_t(B)$ 
as the number of deaths that occured in $B$ up to time $t >0$, i.e.
\[
D_t(B):= \Bigl\lvert\bigl\{s \in (0,t] \bigm\vert |(\eta_{s-}\setminus \eta_{s} )\cap B| = 1  \bigr\}\Bigr\rvert.
\]
Then, for a fixed $B \in \mathscr{B}_b(\R ^\d)$, $(D_t(B))_{t\geq 0 }$
is an a.s. finite increasing $\{\mathscr{F}_t, t \geq 0 \}$-adapted process, 
in particular, it is a sub-martingale. Moreover,
\begin{equation}\label{eqqwwrq324}
D_t(B)\leq \eta _0 (B) + \Pi _t(B).
\end{equation}

In the sequel, we will need the assumption that $\E \eta _0 (B)<\infty$ for all $B\in\mathscr{B}_b(\R^\d)$. Note that if $\E W(\eta _0)<\infty$ then the former inequality always holds as we assumed that $h$ is separated from $0$ on each compact set, and hence $\E \eta _0 (B)\leq \alpha \E \langle h,\eta_0\rangle \leq \alpha \E W(\eta _0)<\infty$ for some $\alpha=\alpha(B)>0$.

Under this assumption, \eqref{eqqwwrq324} implies that, for a fixed $B \in \mathscr{B}_b(\R ^\d)$, the process $(D_t(B))_{t\geq 0 }$ is uniformly integrable on finite time intervals. Therefore by Doob--Meyer decomposition theorem there exists a unique
predictable increasing process $(A_t(B))_{t\geq 0 }$ such that 
$D_t(B) - A_t(B)$ is a martingale.

\begin{lem}\label{compensator}
Let $B \in \mathscr{B}_b(\R ^\d)$ and $\E \eta _0 (B)<\infty$. Then
	\begin{equation}
		A_t(B) = \int\limits_0 ^t \eta _{s-}(B)ds, \quad t\geq0. \label{A-compens}
	\end{equation}
\end{lem}
\begin{proof}
By \eqref{se GK}, $D_t(B)=S_1 (t)+S_2 (t)$, where
\begin{align*}
S_1(t):&= \tilde N\Big( \bigl\{ (x,s,r,u)\mid s \in (0,t ], u \leq b(x, \eta _{s-}), x \in B, r \leq t-s \bigr\} \Big)\\
&=\tilde N\Big( \bigl\{ (x,s,r,u)\mid s \in (0,t ], x \in \eta _s, x \in B, r \leq t-s \bigr\} \Big),\\
S_2(t):&=\tilde{\eta} \Big( \bigl\{ (x,r)\mid  x \in B, r \leq t-s \bigr\} \Big)=\tilde{\eta} \Big( \bigl\{ (x,r)\mid  x \in B, r \leq t-s \bigr\} \Big)
\end{align*}
Both $S_1$ and $S_2$ are, evidently, increasing processes. 

We are going to show firstly that	
\begin{equation}\label{clay}
	S_1(t) - \int\limits_0 ^t \bigl\lvert(\eta _{s-}\setminus \eta_0) \cap B \bigr\rvert ds
	\end{equation}
	is a martingale. To this end we write 
  \begin{align*}
  &\quad S_1 (t + \Delta t) - S_1 (t ) \\
  &= 
  \tilde N\Big( \bigl\{ (x,s,r,u)\mid s \in (0,t ], x \in (\eta _s\setminus \eta _0), x \in B,  t-s < r \leq t + \Delta t - s \bigr\} \Big)
  \\&\quad 
	+ \tilde N\Big( \bigl\{ (x,s,r,u)\mid s \in (t,t + \Delta t ], x \in (\eta _s\setminus \eta _0), x \in B,  r \leq t + \Delta t - s \bigr\} \Big)
	\\&
	=: S_3(t, t + \Delta t) + S_4(t, t + \Delta t).
	\end{align*}
	Note that $S_3(t, t + \Delta t)$ is the number of particles born during $(0,t]$ dying during $(t, t+ \Delta t]$,
	and  $S_4(t, t + \Delta t)$ is the number of particles who both are being born and die during $(t, t+ \Delta t]$.
	Since the lifespan of every particle is a unit exponential, for every $T >0$ a.s.
	\begin{gather*}
	 \lim\limits _{\Delta t \to 0} \sup\limits_{t \in [0,T]} S_4(t, t + \Delta t) = 0,\\
\shortintertext{and also}
	 \PP{ S_4(t, t + \Delta t) > 0 } = o(\Delta t), \quad t > 0.
	\end{gather*}
	For $x \in \eta_t \setminus \eta _0 $, the `residual clock times' (see also \cite{GarciaKurtz}) 
	$r-(t-s)$ are independent of $\mathscr{F}_t$ by the properties of a Poisson point process,
	and hence the residual clock times have the same unit exponential distribution.
	Therefore,
	conditionally on $\mathscr{F}_t$, $S_3(t, t + \Delta t)$ has 
	the binominal distribution with parameters $|(\eta_t\setminus \eta _0) \cap B|$ and $1 - e^{-\Delta t}$. Consequently, 
	\begin{align} 
	\Pc{S_1 (t + \Delta t) - S_1 (t ) = 1 }{\mathscr{F}_t} & = |(\eta_t\setminus \eta _0) \cap B| \Delta t + o(\Delta t), \label{S1-1} \\
	\Pc{S_1 (t + \Delta t) - S_1 (t ) = 0 }{\mathscr{F}_t} & = 1 - |(\eta_t\setminus \eta _0) \cap B| \Delta t + o(\Delta t),  \\
	\Pc{S_1 (t + \Delta t) - S_1 (t ) > 1 }{\mathscr{F}_t} & = o(\Delta t), \label{S1-3} 
	\end{align}
	
	Thus, $(S_1 (t))_{t\geq0}$ is a pure jump type process with unit jumps, and it follows from 
	\eqref{S1-1}-\eqref{S1-3} that the rates of jumps at time $t$ are given by $ |(\eta_t\setminus \eta _0) \cap B|$.
	Hence the process in \eqref{clay} is indeed a martingale. 
	
	Similarly one can show that $S_2 (t) - \int\limits _{0} ^t |\eta_{s-}\cap \eta _0 \cap B| ds$ is a martingale.
	Therefore, 
	\[
		D_t(B) - \int\limits _0 ^t \eta _{s} (B)ds = D_t(B) - \int\limits _0 ^t \eta _{s-} (B)ds
	\]
	is also a martingale.
	\end{proof}
	
\begin{rmk}It follows from Lemma \ref{compensator} that the point process $D$
	on $\R _+ \times \R ^\d$ defined by
	 $(t,x) \in D \Leftrightarrow  x \in \eta _{t-} \setminus \eta_t $ can be viewed as a point process with the predictable compensator given by \eqref{A-compens}.
\end{rmk}	

	\begin{prop}\label{it is a martingale}
		Suppose that $\E W(\eta _0) < \infty$.
		The process
		 	\begin{equation}\label{shoehorn}
		 M_t: = W(\eta _t) - \int\limits _{0} ^t LW (\eta _s) ds
		 \end{equation}
		 is an $(\mathscr{F}_t)$-martingale.
	\end{prop}
	\begin{proof} 
	Let $B \in\mathscr{B}_b(\R ^\d)$. Then
	\begin{align*}
	&\quad	W(\eta _t \cap B) -  \ W  (\eta _0) \\
	&= \int\limits _{(0,t] \times B \times \R _+ } I _{ [0,b(x,\eta _{s-} )] } (u) \bigl\{W(\eta _{s-} \cup x) - W(\eta _{s-}) \bigr\}
	 N_b(ds, dx, du)
\\&\quad
	+ \int\limits _{(0,t] \times B} \bigl\{W(\eta _{s-} \setminus x) - W(\eta _{s-}) \bigr\} D(ds, dx)
	\\&= \int\limits _{(0,t] \times B \times \R _+ } I _{ [0,b(x,\eta _{s-} )] } (u) \bigl\{W(\eta _{s-} \cup x) - W(\eta _{s-}) \bigr\}
	\big( N_b(ds, dx, du) - dsdxdu \big)
\\&\quad
	 + \int\limits _{(0,t] \times B} b(x,\eta _{s-} ) 
	 \bigl\{W(\eta _{s-} \cup x) - W(\eta _{s-}) \bigr\} ds dx 
\\&\quad
	+ \int\limits _{(0,t] \times B} \bigl\{W(\eta _{s-} \setminus x) - W(\eta _{s-}) \bigr\}  \big( D(ds, dx) - ds \eta _{s-}(dx) \big)
\\&\quad
	+ \int\limits _{(0,t] \times B} \bigl\{W(\eta _{s-} \setminus x) - W(\eta _{s-}) \bigr\}  ds \eta _{s-}(dx).
	\end{align*}
	By Lemma \ref{compensator},	the integrals with respect to  $N_b(ds, dx, du) - dsdxdu$ and $D(ds, dx) - ds \eta _{s-}(dx)$ are martingales
	as integrals with respect to the difference between a point process and its compensator, see e.g.
	  \cite[(3.8), Chapter 2]{IkedaWat}.
	  Therefore,
	\begin{align*}
	W(\eta _t \cap B) & - \int\limits _{(0,t] \times B} b(x,\eta _{s-} ) \bigl\{W(\eta _{s-} \cup x) - W(\eta _{s-}) \bigr\} ds dx 
	\\
	 & - \int\limits_{(0,t]} ds \sum\limits _{x \in \eta _{s-}\cap B} \bigl\{W(\eta _{s-} \setminus x) - W(\eta _{s-}) \bigr\}
	\end{align*}
	is indeed an $(\mathscr{F}_t)$-martingale. The statement of the Lemma follows then from the dominated convergence theorem by using Proposition~\ref{propnew} and Corollary~\ref{slink}.
\end{proof}

	\begin{defi}
		We will call a function $F : \Theta \to \R _+$
		a Forster--Lyapunov, or an FL function, if
		there exist $M,c>0$ such that, cf. \eqref{the generator},
		\begin{equation*}
		L F (\eta) \leq M  - cF(\eta), \ \ \ \eta \in  \Theta.
		\end{equation*}
	\end{defi}

	\begin{lem}\label{walk out on}
		The function $W$ is an FL function:
		\begin{equation}\label{FLest}
		LW (\eta)  \leq  \textbf{b}  \langle h\rangle  - \frac 12 W (\eta),  \quad \eta \in \Theta.
		\end{equation}
	\end{lem}
	\begin{proof}
	By \eqref{the generator}, \eqref{twinge1}, \eqref{twinge11}, for each $\eta\in\Theta$, we have
\begin{align*}
LW(\eta)&=- \sum\limits _{z \in \eta} h(z) -2 \sum\limits _{\{x,y\} \subset \eta} \psi(x,y)\\&\quad+
\int\limits _{\R ^\d} b(x,\eta)   h(x)dx  + \sum_{y\in\eta}\int\limits _{\R ^\d} b(x,\eta)  \psi(x,y)  dx.
\end{align*}
Then, by \eqref{b_is_bdd}, \eqref{profligate}, \eqref{bland},
	\begin{align*}
LW(\eta)&\leq - \sum\limits _{z \in \eta} h(z) -2 \sum\limits _{\{x,y\} \subset \eta} \psi(x,y)+
\textbf{b}\langle h\rangle  +C_1\textbf{b} \sum_{y\in\eta} \phi(y)\\
&\leq \textbf{b}\langle h\rangle- \frac{1}{2}\sum\limits _{z \in \eta} h(z) -2 \sum\limits _{\{x,y\} \subset \eta} \psi(x,y),
\end{align*}
that yields \eqref{FLest}.
		\end{proof}

	\begin{thm}\label{prop1}
	Suppose that  $\E W(\eta _0) < \infty $. Then
		\begin{equation}\label{limsupest}
		\limsup\limits _{t \to \infty} \E W(\eta _t) \leq 2 \textbf{b} \langle h\rangle.
		\end{equation}
	\end{thm}
	\begin{proof}
	By Proposition~\ref{propnew}, $w(t) : = \E W (\eta _t)<\infty$ for $t>0$.
	By Proposition~\ref{it is a martingale}, \eqref{shoehorn} defines a martingale. Taking the expectation in \eqref{shoehorn}, we get
	\[
	 w(t) = \E \int\limits _{0} ^t  LW (\eta _s) ds +w (0).
	\]
	  Hence $w$ is differentiable. Taking the derivative, we obtain
	  \[
	   w'(t) = \E LW (\eta _t) \leq  \textbf{b} \langle h\rangle- \frac12 w(t),
	  \]
by \eqref{FLest}. By the comparison principle,
	 \[
	 w(t)\leq e^{-\frac12 t}w(0)+\textbf{b}\langle h\rangle \int_{0}^t e^{-\frac12 (t-s)}ds\leq e^{-\frac12 t}w(0)+2\textbf{b}\langle h\rangle,
	 \]
	 that yields \eqref{limsupest}.
	 \end{proof}

	Now we are going to apply the techniques similar to the considered in \cite{MeynTweedie3}.	Let $\delta >0$ be a small number.
	For $K >0$, let $\tau _K$ be the return times to the set $\{\zeta \in \Theta: W(\zeta)< K\}$, namely,
	\begin{equation}\label{deftauK}
	\tau _K = \inf\bigl\{t > \delta \bigm\vert W(\eta _t)< K \bigr\}
	\end{equation}

	\begin{prop}
		Suppose that 	$\E  W (\eta _{0}) < \infty$. Then, for each	$K >\textbf{b}  \langle h\rangle $, 
	\begin{equation*}
		\E \tau _{K} \leq \frac{\E  W (\eta _{0})}{K - \textbf{b}  \langle h\rangle}.
		\end{equation*}
	\end{prop}
\begin{proof}
	 Denote $\kappa = K - \textbf{b}  \langle h\rangle >0$.
	 By \eqref{FLest} and \eqref{deftauK}, we have a.s. on $\{t < \tau _{K }\}$,
	 \begin{equation}\label{estforsmalltime}
	 LW (\eta _t) \leq \textbf{b}  \langle h\rangle- K = -\kappa.
	 \end{equation}
Next, by \eqref{shoehorn}, 
	\[
	\E W (\eta _{\tau _{K } \wedge t }) - \E \int\limits _{0} ^{\tau _{K } \wedge t } L W(\eta _{s}) ds = \E  W (\eta _{0}),
	\]
    hence, by \eqref{estforsmalltime},
	\[
	\E  W (\eta _{0}) \geq
	\E W (\eta _{\tau _{K } \wedge t }) + \kappa \E\tau _{K } \wedge t 
	\]
	and, therefore,
	\[
	 \kappa \E\tau _{K } \wedge t  \leq \E  W (\eta _{0}) - \E W (\eta _{\tau _{K } \wedge t })  \leq \E  W (\eta _{0}).
	\]
	Taking  $t \to \infty$,  we get the desired result.
\end{proof}

The next proposition shows the existence of an exponential moment
of   $\tau _K$ for sufficiently large $K$.
	\begin{prop}\label{prop:expmoments}
 Assume that $\E  W (\eta _{0}) < \infty$. Then,
   for all $\theta \in (0,\frac 12)$, there exists $K_\theta>0$ such that, for $K>K_\theta$,
	\begin{equation*}
	\E e^{\theta \tau _{K}} < \infty.
	\end{equation*}
\end{prop}
\begin{proof}
Fix any $\theta \in (0, \frac 12)$, and define
\[
\Phi (t,x) : = e ^{\theta t}(x + 1)\geq 1, \quad t,x\in\R_+.
\]

We are now going to show that
\begin{equation}\label{faucet}
\Phi (t, W (\eta _t)) - \int\limits _0 ^t \mathbb{L} \Phi (s, W (\eta _s)) ds
\end{equation}
is a local martingale, where
\begin{equation*}
\mathbb{L} \Phi (s, W (\eta _s)) := \theta \Phi (s, W (\eta _s)) + e^{\theta s} LW (\eta _s) .
\end{equation*}
By  \cite[Proposition 3.2, Chapter 2]{EthierKurtz} and since $e^{\theta t}$ is locally of bounded variation, the process
\begin{equation} \label{maroon}
\begin{split}
\widetilde{M}_t:= & e ^{\theta t}\left[W (\eta _t) + 1 - \int\limits _0 ^t LW (\eta _s) ds \right]
\\
& - \int\limits _{0} ^t \left[ W (\eta _s) + 1 - \int\limits _0 ^s LW (\eta _u) du \right] \theta e^{\theta s}ds
\end{split}
\end{equation}
is a local martingale.
Now,
\begin{equation*}
\begin{gathered}
e ^{\theta t} \int\limits _0 ^t LW (\eta _s) ds - \int\limits _{0} ^t \int\limits _0 ^s LW (\eta _u) du \theta e^{\theta s}ds
\\
=  e ^{\theta t} \int\limits _0 ^t LW (\eta _s) ds -
\int\limits _{0} ^t  LW (\eta _u) du \int\limits _{u} ^t \theta e^{\theta s}ds
\\
= e ^{\theta t} \int\limits _0 ^t LW (\eta _s) ds -
\int\limits _{0} ^t  LW (\eta _u) (e ^{\theta t} - e ^{\theta u}) du =
\int\limits _{0} ^t   e ^{\theta u}  LW (\eta _u) du,
\end{gathered}
\end{equation*}
hence from \eqref{maroon}
\begin{align*}
\widetilde{M}_t &=  e ^{\theta t}\left[W (\eta _t) + 1  \right]
 - \int\limits _{0} ^t \left[ W (\eta _s) + 1 \right] \theta e^{\theta s}ds
 - \int\limits _{0} ^t   e ^{\theta s}  LW (\eta _s) ds
 \\
 &= \Phi (t, W (\eta _t)) - \int\limits _{0} ^t \theta \Phi (t, W (\eta _s))ds - \int\limits _{0} ^t   e ^{\theta s}  LW (\eta _s) ds,
\end{align*}
and we see that the process in \eqref{faucet} is indeed a local martingale.

By Lemma~\ref{walk out on},
\[
\mathbb{L} \Phi (s, W (\eta _s)) \leq
e^{\theta s}\left[\theta W (\eta _s)  + \theta + \textbf{b} \langle h\rangle - \frac 12 W (\eta _s)    \right].
\]
Take a sequence of stopping times $\{\sigma _n \}$, such that $\sigma _n \nearrow \infty$, $n\to\infty$, a.s. Then
\begin{align*}
&\quad \E e^{\theta \tau _{K } \wedge t \wedge \sigma _n} \\
& \leq \E e^{\theta \tau _{K } \wedge t \wedge \sigma _n} [W(\eta _{e^{\theta \tau _{K } \wedge t \wedge \sigma _n}} ) +1]\\&
= \E \Phi (\tau _{K } \wedge t \wedge \sigma _n, W (\eta _{\tau _{K } \wedge t \wedge \sigma _n}))
\\
& = \E\Phi (0, W(\eta _0)) + \E\int\limits _{0} ^{\tau _{K } \wedge t \wedge \sigma _n} \mathbb{L} \Phi (s, W (\eta _s))  ds
\\ &
\leq \E W(\eta _0) + \E\int\limits _{0} ^{\tau _{K } \wedge t \wedge \sigma _n}
e^{\theta s}\left[\theta W (\eta _s)  + \theta + \textbf{b} \langle h\rangle - \frac 12 W (\eta _s)    \right] ds
\\
& \leq \E W(\eta _0) + \E\int\limits _{0} ^{\tau _{K } \wedge t \wedge \sigma _n}
e^{\theta s}\biggl(\Bigl(\theta-\frac 12\Bigr) K   + \theta + \textbf{b} \langle h\rangle   \biggr) ds.
\end{align*}
Therefore, for $K > \bigl(\theta + \textbf{b}  \langle h\rangle\bigr)\bigl(\frac 12 - \theta\bigr) ^{-1}$ we get
\begin{equation*}
 \E e^{\theta \tau _{K } \wedge t \wedge \sigma _n}  \leq  \E W(\eta _0).
\end{equation*}
Taking here $n \to \infty$ and then $t \to \infty$ concludes the proof.
\end{proof}

\section{Uniqueness of the degenerate invariant distribution for sublinear birth rate}\label{sec:uniqinv}

In this Section, we study some general properties of birth-and-death processes which are described through Definition~\ref{defBaD} and which have sublinear birth rate $b$, namely,
\begin{equation}\label{sublin}
b(x,\eta)\leq \sum_{y\in\eta} g(x-y), \quad \eta \in\Gamma_G,\ x\notin\eta,
\end{equation}
for some $g:\R^\d\to(0,\infty)$, such that $g(x)\leq BG(x)$, $x\in\R^\d$, with some $B>0$. We will assume that $(\eta_t)_{t\geq0}$ is the unique solution to \eqref{se GK} in the sense of Definition~\ref{def:sol}.

We are going to find sufficient conditions for $g$ such that
\begin{equation}\label{smallg}
\langle g\rangle<1
\end{equation}
would imply that the Dirac measure concentrated at $\varnothing$
is the only invariant distribution for $\eta _t$ on $\Gamma_G$.

Note that, \eqref{sublin} implies $b(x,\varnothing)=0$ and hence the empty configuration is a trap. Therefore, the Dirac measure concentrated at $\varnothing$ is indeed an invariant distribution for $(\eta _t)$, so one need to show the uniqueness only.

Again, our first example is the rate given by \eqref{brate fecundity} under Conditions~\ref{con1}--\ref{con3}. Then Condition~\ref{con3} implies \eqref{sublin} with
\begin{align}
g(x)&=r_p a(x), \quad x\in\R^\d,\label{gisrpa}\\
\shortintertext{where}
r_p&:=\sup_{s\in\R_+} (1+ps)e^{-s}=\begin{cases}
1, & 0\leq p\leq 1,\\
p e^{\frac{1}{p}-1}, & p>1,
\end{cases}\label{rpconst}
\end{align}
i.e. \eqref{smallg} takes the form
\[
	r_p \langle a\rangle <1.
\]

Another example is the establishment birth rate \eqref{brate extablishment}. The condition $c(x)\leq p \phi(x)$, $x\in\R^\d$, with $p\geq0$, cf.~\eqref{estkernest}, implies that the surviving rate \eqref{survrate} is bounded by \eqref{rpconst}, and hence \eqref{sublin} also holds with $g$ given by \eqref{gisrpa}.

Finally, an evident example is $b(x,\eta)=\sum_{y\in\eta} g(x-y)$. The corresponding process is then a special form of a spatial branching (nonlocal) process (when each particle $y$ may die and produce zero or two off-springs: one of them is at the same position  $y$ and another is distributed according to the kernel $g$), also known as a contact process in the continuum, see \cite{KonSkor,Dur1979}. This rate is not bounded and does not satisfy the assumption \eqref{b_is_bdd} of Theorem~\ref{thm_exist_uniq}, however, it is straightforward to check that it satisfies the condtions of \cite[Theorem~2.13, Lemma~3.14]{GarciaKurtz} and hence the statement of Theorem~\ref{thm_exist_uniq} holds true for it. Note that then the assumption \eqref{smallg} describes the so-called sub-critical regime when the process `dies out', see e.g. \cite[Section~3]{Dur1979}.

\begin{thm}
Let \eqref{sublin} hold with 
\begin{equation}
g(x)=b_g(|x|)\leq C(1+|x|)^{-\d-2\varepsilon}, \quad x\in\R^\d, \label{b-est0}
\end{equation}
for some $C>0$, where
 $b_g:\R_+\to(0,\infty)$ is a continuously decreasing to $0$ function. Suppose also that \eqref{smallg} holds. 
Then the Dirac measure concentrated at $\varnothing$
is the only invariant distribution for $\eta _t$ on $\Gamma_G$.
\end{thm}
\begin{proof}
Recall, that we have to show that the uniqueness of the invariant distributnon only. Note also that \eqref{b-est0} yields $g(x)\leq C G(x)$, $x\in\R^\d$. 

Let $\sigma_d$ denote the surface area of a unit sphere in $\R^\d$; then, by \eqref{b-est0} and \eqref{smallg}, 
\[
	\langle g\rangle =\sigma_d\int_0^\infty b_g(s)s^{\d-1}ds<1.
\]

Choose $R>0$, such that $C\leq (1+R)^{\frac{1}{2}\varepsilon}$ and also
\begin{equation}\label{choiceR}
\sigma_d\int_{R}^\infty (1+s)^{-\d-\frac32\varepsilon}s^{\d-1}ds<1- \langle g\rangle.
\end{equation}
Then, for $|x|\geq R$,
\[
b_g(|x|)\leq \frac{C}{(1+|x|)^{\d+2\varepsilon}}\leq \frac{(1+R)^{\frac12\varepsilon}}{(1+|x|)^{\frac12\varepsilon}}\frac{1}{(1+|x|)^{\d+\frac32\varepsilon}} \leq \frac{1}{(1+|x|)^{\d+\frac32\varepsilon}}.
\]
Let $R_1\geq R$ satisfy
\[
(1+R_1)^{-\d-\frac32\varepsilon}=b_g(R).
\]
Set
\[
c_g(x)=\begin{cases}
b_g(|x|), & |x|\leq R,\\[2mm]
b_g(R), & R\leq |x|\leq R_1,\\
(1+|x|)^{-\d-\frac32\varepsilon}, & |x|\geq R_1.
\end{cases}
\]
As a result,  $c_g:\R^\d\to(0,\infty)$ is a radially symetric  continuous bounded integrable function, such that
\begin{equation}\label{c-est}
	g(x)= b_g(|x|)\leq c_g(x)\leq \begin{cases}
b_g(|x|), & |x|\leq R,\\[2mm]
(1+|x|)^{-\d-\frac32\varepsilon}, & |x|\geq R,
\end{cases}
\end{equation}
and therefore, by \eqref{choiceR}
\begin{align}
\langle c_g \rangle &\leq\int_{\{|x|\leq R\}}b_g(|x|)dx+\int_{\{|x|\geq R\}}\dfrac{1}{(1+|x|)^{\d+\frac32\varepsilon}}dx \notag \\&\leq \langle g\rangle+ \sigma_d\int_{R}^\infty (1+s)^{-\d-\frac32\varepsilon}s^{\d-1}ds<1.
\label{cgislessthan1}
\end{align}

Let $*$ denote the convolution of functions over $\R^\d$. Set, for each $n\in\N$, $c_g^{*n}(x)=(c_g*\ldots*c_g)(x)$ ($n-1$ times).
It is straightforward to check that the normalised function $\langle c_g \rangle^{-1} c_g$ satisfies the assumptions of \cite[Theorem 4.1]{FT2017b}. Then there exists
$\alpha_0 \in (0, 1)$ such that, for any $\delta \in (0, 1)$ and $\alpha\in (\alpha_0, 1)$, there exist $B_1 = B_1(\delta,\alpha) > 0$ and $\lambda = \lambda(\delta,\alpha) \in (0,1)$, such that, for all $n\in\N$,
\begin{equation}\label{Kestenbound}
c_g^{*n}(x)\leq B_1\langle c_g \rangle^n (1+\delta)^n \min\{\lambda,c_g(|x|)^\alpha\}, \quad x\in\R^\d.
\end{equation}

By \eqref{cgislessthan1}, one can fix a $\delta\in(0,1)$ such that
\begin{equation}\label{cg1plusdeltaislessthan1}
\langle c_g \rangle (1+\delta)<1.
\end{equation}
Choose also some $\alpha\in (\alpha_0, 1)$ such that
\[
\alpha\Bigl(\d+\frac{3}{2}\varepsilon\Bigr)>\d+\varepsilon.
\]
Then, by \eqref{c-est},
\begin{equation}\label{cgleqKG}
c_g(|x|)^\alpha\leq K G(x), \quad x\in\R^\d
\end{equation}
for some $K>0$.

By \eqref{Kestenbound}, \eqref{cg1plusdeltaislessthan1}, \eqref{cgleqKG}, the series
\[
f(x):=c_g(x)+\sum_{n=2}^\infty  c_g^{*n}(x), \quad x\in\R^\d
\]
converges uniformly on each compact of $\R^\d$ and its sum $f$ is a continuous function such that $f(x)\leq K_1 G(x)$, $x\in\R^\d$, for some $K_1>0$. In particular, of course, $f$ is integrable and bounded on $\R^\d$. Moreover, $f$ satisfies the equality
\begin{equation}\label{repres}
f(x)=c_g(x)+(c_g*f)(x), \quad x\in\R^\d.
\end{equation}

Set
\[
F(\eta): = \langle f, \eta \rangle, \quad \eta\in\Gamma_G.
\]
Recall that $L$ is given by \eqref{the generator}, then, for $\eta\in\Gamma_G$, \eqref{sublin}, \eqref{c-est}, and \eqref{repres} imply that
	\begin{align}
	L F(\eta) &\leq -
	\langle f,  \eta \rangle+ \sum_{y\in\eta}\int_{\R^\d}g(x-y)f(x) dx\notag
	\\ &\leq
	\langle c_g \ast f, \eta \rangle -
	\langle f,  \eta \rangle
	= - \langle c_g, \eta \rangle. \label{quail}
	\end{align}

Assume firstly that $\E F(\eta _0) < \infty$. Similarly to the proof of Proposition~\ref{it is a martingale}, one can show that $F(\eta_t)-\int_0^t LF(\eta_{s-})ds$ 	is a  martingale. 
Then \eqref{quail} implies that $ F(\eta _t) +	 \int\limits _{0} ^t \langle c_g, \eta _{s-} \rangle ds$ is a non-negative supermartingale.
	 Hence, by Doob's martingale convergence theorem, a.s.
	\begin{equation*}
	  \int\limits _{0} ^\infty \langle c_g, \eta _{s-} \rangle ds < \infty.
	\end{equation*}
	Furthermore,
		\begin{equation}\label{winnow2}
	 \E \int\limits _{0} ^\infty \langle c_g, \eta _{s-} \rangle ds  = \lim\limits _{t \to \infty}  \E \int\limits _{0} ^t \langle c_g, \eta _{s-} \rangle ds
	 \leq \E F(\eta _0) < \infty.
	\end{equation}
	Consequently, if $\mu$ is an invariant  distribution for $(\eta _t)_{t\geq 0}$ satisfying
	\begin{equation*}
	\int\limits _{\Gamma_G} \langle f,\eta \rangle \mu(d\eta) < \infty,
	\end{equation*}
	then $\mu$ is the Dirac measure   concentrated at the empty configuration 
		since otherwise $\E \int\limits _{0} ^\infty \langle c_g, \eta _{s-} \rangle ds  = \infty$ if $\eta _0$ 
		is distributed according to $\mu$, contradicting to  \eqref{winnow2}.
	
 We have shown that, apart from the delta measure at $\varnothing$, there is no invariant distribution $\mu$
		with $\int\limits _{\Gamma_G} \langle f,\eta \rangle \mu(d\eta) < \infty$.	 
		Assume that there exists an invariant distribution $\mu$ with $\int\limits _{\Gamma_G} \langle f,\eta \rangle \mu(d\eta) = \infty$. 

	 Let $\eta _0$ be distributed according to $\mu$.
	Define
\begin{equation*}
        \eta _0 ^M = \begin{cases}
        \eta _0,   & \text{ if } \langle f,\eta _0 \rangle \leq M ,
        \\
        \varnothing, & \text{ otherwise. }
        \end{cases}
        \end{equation*} 

Consider now the process $(\eta _t ^M)_{t\geq 0}$ started from $\eta _0 ^M$.
	By the uniquness of solutions to \eqref{se GK},
	on  the event $\{ \eta _0 ^M = \eta _0 \}$ we have a.s.
	$\{ \eta _t ^M = \eta _t, t \geq 0 \}$.
	From \eqref{winnow2} it follows that,
	for every $\varepsilon >0$,
			\begin{equation*}
		\inf_{ t\geq 0} \PP{ \langle c_g, \eta _{t-} ^M \rangle > \varepsilon} \leq 
\liminf_{t \to \infty} \frac{\E \langle c_g, \eta _{t-} ^M \rangle }{\varepsilon} \leq
\liminf_{t \to \infty} \frac{\E F(\eta_0) }{t\varepsilon}
		   = 0,
		\end{equation*}
	hence 
	\begin{gather*}
	  \P \{ \langle c_g, \eta _{0} \rangle > \varepsilon \}  =
	\inf  \limits _{t \geq 0}  \P \{ \langle c_g, \eta _{t-} \rangle > \varepsilon \}
	\\
	 \leq \liminf  \limits _{M \to \infty}\Big( \P \{ \eta _0 ^M \ne \eta _0  \} +
	 	\inf_{t \geq 0}  \P \{ \langle c_g, \eta _{t-} ^M \rangle > \varepsilon \}\Big)
	 	= 0,
	\end{gather*}
which contradicts to $\E \langle c_g, \eta _{0} \rangle  = \infty$ and therefore  completes the proof of the proposition.
		
\end{proof}

\end{document}